\documentclass[11pt]{amsart}

\usepackage{amsmath}
\usepackage{amssymb}
\usepackage{amsfonts}
\usepackage{amsthm}  
\usepackage{bm} 
\usepackage{latexsym}
\usepackage{graphicx}    
\usepackage{xypic} 
\usepackage{tikz} 
\usepackage{color} 
\usepackage{mathdots}
\usepackage[pdftex,colorlinks,linktocpage=true,citecolor=blue,linkcolor=blue]{hyperref}
\usepackage[nameinlink]{cleveref}

\crefname{thm}{theorem}{theorems}
\crefname{lem}{lemma}{lemmas}
\crefname{cor}{corollary}{corollaries}
\crefname{Prop}{proposition}{propositions}
\crefname{defn}{definition}{definitions}
\crefname{eg}{example}{examples}
\crefname{xca}{exercise}{exercises}
\crefname{conj}{conjecture}{conjectures}
\crefname{rmk}{remark}{remarks}
\crefname{qst}{question}{questions}
\crefname{obs}{observation}{observations}

\newtheorem{thm}{Theorem}[section]
\newtheorem{lem}[thm]{Lemma}  

\newtheorem{Prop}[thm]{Proposition}
\newtheorem{defn}[thm]{Definition}

\newtheorem{conj}[thm]{Conjecture}
\newtheorem{rmk}[thm]{Remark}

\numberwithin{equation}{section}

\newcommand{\abs}[1]{\lvert#1\rvert}
\newcommand\C{\mathbb C}    
\newcommand\R{\mathbb R}    

\newcommand\style{\mathcal}
\newcommand\tran{{}^t} 
\newcommand\st{\operatorname{St}} 
\newcommand\GL{\mathbf{GL}} 
\newcommand\SL{\mathrm{SL}} 
\newcommand\G{\mathbf{G}} 
\newcommand\SO{\mathbf{SO}} 
\newcommand\Hbf{\mathbf{H}} 
\newcommand\X{\mathbf{X}} 
\newcommand\Abf{\mathbf{A}} 
\newcommand\Pbf{\mathbf{P}} 
\newcommand\Mbf{\mathbf{M}} 
\newcommand\wf{\mathcal W_F} 
\newcommand\lf{\mathcal L_F} 
\newcommand\wt{\widetilde} 

\DeclareMathOperator{\ran}{Im} 
\DeclareMathOperator{\ind}{Ind} 
\DeclareMathOperator{\Hom}{Hom} 
\DeclareMathOperator{\Sym}{Sym} 

\begin{document}
\title{Linear periods and distinguished local parameters}
\author[J.~M.~Smith]{Jerrod Manford Smith}
\address{Department of Mathematics \& Statistics, University of Calgary, Calgary, Alberta, Canada, T2N 1N4}
\email{jerrod.smith@ucalgary.ca}
\urladdr{}
\thanks{}

\date{\today}
\dedicatory{}
\begin{abstract}
Let $F$ be a nonarchimedean local field of characteristic zero.  Let $X$ be the $p$-adic symmetric space $X = H \backslash G$, where $G = \mathbf{GL}_{2n}(F)$ and $H = \mathbf{GL}_n(F) \times \mathbf{GL}_n(F)$.  We verify a conjecture of Sakellaridis and Venkatesh on the Langlands parameters of certain representations in the discrete spectrum of $X$.
\end{abstract}
\maketitle

\section{Introduction}\label{sec-intro}
Let $F$ be a nonarchimedean local field of characteristic zero.
Let $G = \GL_{2n}(F)$ and let $H = \GL_n(F) \times \GL_n(F)$.  The subgroup $H$ of $G$ is equal to the fixed points of an involution $\theta$ of $G$ and the quotient $X = H\backslash G$ is a $p$-adic symmetric space.
We are interested in the study of harmonic analysis on $X$ and its relevance to the Local Langlands Correspondence (LLC).  
In particular, we aim to understand the discrete spectrum of $X$, as a representation of $G$, within the framework developed by Sakellaridis and Venkatesh \cite{sakellaridis--venkatesh2017}. 
The goal of this note is to understand the construction of relative discrete series for $X$ via \cite[Theorem 6.3]{smith2018}  (see \Cref{thm-smith}) in terms of the associated Langlands parameters, and to provide evidence for the truth of the general conjectures on the discrete spectrum formulated by Sakellaridis and Venkatesh \cite[Conjecture 16.2.2]{sakellaridis--venkatesh2017}.

Let $(\pi,V)$ be an admissible representation of $G$ on a complex vector space $V$.  
If there is a nonzero element $\lambda \in \Hom_H(\pi,1)$, then we say that $(\pi,V)$ admits a (local) linear period and we refer to the $H$-invariant linear form $\lambda$ on $V$ as a linear period.  
This terminology parallels the language used in the global setting.
If $(\pi,V)$ admits a nonzero linear period then we say that $\pi$ is $H$-distinguished.
It is exactly the $H$-distinguished representations of $G$ that are relevant to the harmonic analysis on $X$.
The irreducible direct summands of the space $L^2(X)$ of square integrable $\C$-valued functions on $X$ are referred to as relative discrete series representations.
If $(\pi,V)$ is an $H$-distinguished discrete series representation of $G$, then $(\pi,V)$ is automatically a relative discrete series representation for $X$ \cite{kato--takano2010}.
The relative discrete series representations constructed in \cite[Theorem 6.3]{smith2018} consist of certain tempered representations of $G$ that do not appear in the discrete spectrum of the group $G$.

For $H$-distinguished supercuspidal representations of $G$, the main result of this paper is due to Jiang, Nien and Qin \cite{jiang--nien--qin2008,jiang--nien--qin2010}.   We refer the reader to \cite[Theorem 1.1]{jiang--nien--qin2010} for a complete statement of their characterization of $H$-distinction for supercuspidal representations of $G$ (in terms of $L$-functions, local Langlands functorial transfer, and the Shalika model).
Our work relies heavily on the results of Jiang, Nien and Qin and the work of Matringe \cite{matringe2014a}.

We now take a moment to briefly outline the contents of the paper.  In the next section, we fix notation and conventions.  
We recall the notion of a distinguished parameter and the relevant conjecture of Sakellaridis--Venkatesh in \Cref{sec-dist-parameters}.
Linear periods and the discrete spectrum of $X$ are discussed in \Cref{sec-rds}; in this section, we recall several known results including the work of Matringe \cite{matringe2014a} and the author's construction of relative discrete series \cite{smith2018}.
In the final part (\Cref{sec-parameters}), we study the $L$-parameters of the relative discrete series studied in \cite{smith2018} and prove the main result \Cref{thm-rds-dist-parameter}.

\subsection{Notation}\label{sec-notation}
Let $\wf$ be the Weil group of $F$.  The local Langlands group of $F$ is the Weil-Deligne group $\lf = \wf \times \SL(2,\C)$. 
For any integer $k \geq 1$, denote by $\style{S}(k)$ the unique (up to equivalence) $k$-dimensional irreducible representation of $\SL(2,\C)$.  Recall that $\style{S}(k) \cong \Sym^{k-1}(\C^2)$, where $\C^2$ is the standard representation of $\SL(2,\C)$.
We denote the rank-$k$ general linear group $\GL_k(F)$ by $G_k$.
Let $H_{2k}$ be the subgroup $\GL_k(F) \times \GL_k(F)$ of $G_{2k}$.

Let $P$ be a parabolic subgroup of $G$ with Levi subgroup $M$ and unipotent radical $N$.  
 Given a smooth representation $(\rho, V_\rho)$ of $M$ we may inflate $\rho$ to a representation of $P$, also denoted $\rho$, by declaring that $N$ acts trivially.
 We define the representation $\iota_P^G \rho$ of $G$ to be the (normalized) parabolically induced representation $\ind_P^G (\delta_P^{1/2} \otimes \rho)$. 
We will also use the Bernstein--Zelevinsky \cite{bernstein--zelevinsky1977,zelevinsky1980} notation $\pi_1 \times \ldots \times \pi_{k}$ for the (normalized) parabolically induced representation $\iota_{P_{(m_1,\ldots,m_k)}}^{G_n}(\pi_1\otimes \ldots \otimes \pi_k)$ of $G_n$ obtained from the standard (block-upper triangular) parabolic subgroup $P_{(m_1,\ldots,m_k)}$ and representations $\pi_{j}$ of $G_{m_j}$, where $\sum_{j=1}^k m_j = n$.

Let $\rho$ be an irreducible supercuspidal representation of $G_r$ and let $k \geq 1$ be a positive integer.
Let $\nu$ denote the character of $G_r$ given by $g\mapsto \abs{\det(g)}_F$, where $|\cdot|_F$  is the normalized absolute value on $F$ and $r$ is understood from context.
The unique irreducible quotient $\st(k,\rho)$ of the induced representation
\[
\nu^{\frac{1-k}{2}} \rho \times \ldots \times \nu^{\frac{k-1}{2}} \rho
\]
is a discrete series representation of $G_{kr}$ \cite[Theorem 9.3]{zelevinsky1980}.  We refer to the representations $\st(k,\rho)$ as generalized Steinberg representations. The usual Steinberg representation of $G_n$ is $\st(n,1)$.
Note that $\st({k_1},\rho_1)$ is equivalent to $\st({k_2},\rho_2)$ if and only if $k_1 = k_2$ and $\rho_1$ is equivalent to $\rho_2$ \cite[Theorem 9.7(b)]{zelevinsky1980}.
\subsection{Distinguished parameters}\label{sec-dist-parameters}
In this section, suppose that $\G$ is an arbitrary connected reductive group that is defined and split over $F$.
Let $G^\vee$ be the complex dual group of $\G$.\footnote{Since $\G$ is split over $F$, $\wf$ acts trivially on $G^\vee$ and without loss of generality the $L$-group ${}^LG$ of $\G$ coincides with the dual group $G^\vee$.}
 An $A$-parameter, or an Arthur parameter, for $\G$ is a continuous homomorphism $\psi: \lf\times \SL(2,\C) \rightarrow G^\vee$ such that
 \begin{itemize}
 \item  the restriction $\psi\vert_{\wf}$ of $\psi$ to the Weil group $\wf$ is bounded,
 \item  the image of $\psi\vert_{\wf}$ consists of semisimple elements of $G^\vee$,
 \item  and the restriction of $\psi$ to each of the two $\SL(2,\C)$ factors is algebraic.\footnote{The second $\mathrm{SL}(2,\C)$ factor is referred to as the Arthur $\SL(2)$.}
 \end{itemize}
The $L$-parameter, or Langlands parameter, for $\G$ associated to the $A$-parameter $\psi$ is the admissible homomorphism $\phi_\psi : \lf \rightarrow G^\vee$ given by $\phi_\psi(w,g) = \psi(w,g,d_w)$, where 
\begin{align*}
	d_w = \left( \begin{matrix} |w|^{\frac{1}{2}} & 0 \\ 0 & |w|^{-\frac{1}{2}} \end{matrix}\right) \in \SL(2,\C), 
\end{align*}
and $|\cdot|: \wf \rightarrow \R_{>0}$ is the absolute value determined by the normalized absolute value $|\cdot|_F$ on $F^\times$ via local class field theory.  Langlands parameters of the form $\phi_\psi$ are said to be of Arthur type \cite[Section 3.6]{cunningham--fiori--mracek--moussaoui--xu2018-pp}.

Let $\X$ be a homogeneous spherical variety for $\G$, that is, $\X$ is a (normal) variety defined over $F$ equipped with a transitive $\G$ action such that a Borel subgroup $\mathbf{B}$ of $\G$ has a Zariski-dense orbit.
Note that symmetric $\G$-varieties are spherical.  
Let $\X^\dagger$ denote the open $\mathbf{B}$ orbit in $\X$.  Define $\Pbf_0$ to be the standard parabolic subgroup $\G$ that stabilizes $\X^\dagger$, that is, 
$\mathbf{P}_0 = \{ g \in \G : \X^\dagger = \X^\dagger g \}$. 
Let $x_0 \in \X^\dagger(F)$ and let $\Hbf$ be the stabilizer of $x_0$ in $\G$.
Then $\X \cong \Hbf\backslash \G$.  For a fixed $x_0$ there is a natural choice of Levi subgroup $\Mbf_0$ of $\Pbf_0$ \cite[$\S2.1$]{sakellaridis--venkatesh2017}.

\begin{rmk}
If $\X = \G^\theta \backslash \G$ is a symmetric variety, defined by an $F$-involution $\theta$ of $\G$, then there is a natural choice of $x_0 = \G^\theta \cdot e$.	The parabolic subgroup $\Pbf_0$ is a minimal $\theta$-split parabolic subgroup and the Levi subgroup $\Mbf_0 = \Pbf \cap \theta(\Pbf_0)$ is $\theta$-stable.
Recall that a parabolic subgroup $\Pbf$ of $\G$ is $\theta$-split if $\theta(\Pbf)$ is opposite to $\Pbf$.
\end{rmk}

Before stating their conjectures on the parameters of relative discrete series, we recall Sakellaridis and Venkatesh's definition of a \textit{distinguished morphism}.  We refer the reader to \cite[$\S$2.1--2.2, 3.2]{sakellaridis--venkatesh2017} for more detail and in particular for the definition of the spherical roots.
Fix a maximal $F$-split torus $\Abf_0$ of $\G$ contained in $\mathbf{B}\cap \Mbf_0$.
Let $A^* \subset G^\vee$ be the complex dual torus of $\Abf_0$.
Define $\Abf_X$ to be the torus $\Abf_0 / (\Abf_0\cap\Hbf)$.
Let $A_X^*$ be the complex torus dual to $\Abf_X$.
Dualizing the surjective map $\Abf_0 \rightarrow \Abf_X$ gives rise to a finite-to-one map $A_X^* \rightarrow A^*$.
Let $G_X^\vee$ be the complex dual group of the spherical variety $\X$ defined by Sakellaridis and Venkatesh, defined under the technical assumption of \cite[Proposition 2.2.2]{sakellaridis--venkatesh2017} on the spherical roots.   
The torus $A_X^*$ is a maximal torus of $G_X^\vee$.
The aim of \Cref{def-dist-mor} is to produce an appropriate notion of extending the map $A_X^* \rightarrow A^*$ to a homomorphism from $G_X^\vee$ to $G^\vee$. 
Let $\Sigma_X$ denote the set of spherical roots of $\X$.  Let $\gamma \in \Sigma_X$. It is known \cite{akhiezer1983,brion2001} that either
\begin{itemize}
	\item $\gamma$ is proportional to a postive root $\alpha$ of $\Abf_0$ in $\G$, or;
	\item $\gamma$ is proportional to the sum $\alpha + \beta$ of two positive roots $\alpha$ of $\Abf_0$ in $\G$ such that $\alpha$ is orthogonal to $\beta$ and both $\alpha$ and $\beta$ are contained in a system of simple roots (but not necessarily the simple roots corresponding to $\mathbf{B}$).
\end{itemize}
In the first case, set $\gamma_0 = \alpha$ and in the second case set $\gamma_0 = \alpha+\beta$ (note that in the second case $\alpha + \beta$ is not a root of $\Abf_0$ in $\G$; moreover, $\alpha$ and $\beta$ are not unique, but there is a canonical choice \cite[Corollary 3.1.4]{sakellaridis--venkatesh2017}).
The roots $\alpha$ and $\beta$ are referred to as the associated roots of $\gamma_0$ (if $\gamma_0 = \alpha$, then $\alpha$ is the associated root).
Sakellaridis and Venkatesh call the set $\Delta_X = \{ \gamma_0 : \gamma \in \Sigma_X\}$ the simple normalized spherical roots of $\X$.

\begin{defn}\label{def-dist-mor}
	A distinguished morphism $\xi: G_X^\vee \times \SL(2,\C) \rightarrow G^\vee$ is a group homomorphism such that
	\begin{enumerate}
		\item the restriction of $\xi$ to $G_X^\vee$ extends the canonical map of tori $A_X^* \rightarrow A^*$
		\item for every simple normalized spherical root $\gamma_0 \in \Delta_X$, the corresponding root space of the Lie algebra $\mathfrak{g}_X^\vee$ maps into the sum of root spaces of its associated roots under the differential of $\xi$
		\item the restriction of $\xi$ to the $\SL(2,\C)$ factor is a principal morphism into $M_0^\vee \subset G^\vee$ with weight $2\varrho_{M_0}: \mathbb{G}_m \rightarrow G^\vee$, where $ \mathbb{G}_m$ is identified with the maximal torus of $\SL(2,\C)$ via $a \mapsto \left(\begin{matrix}a&0\\0&a^{-1}\end{matrix}\right)$ and $2\varrho_{M_0}$ is the sum of the positive roots of $\Abf_0$ in $\Mbf_0$.
	\end{enumerate}
\end{defn}

\begin{rmk}
Sakellaridis and Venkatesh proved that distinguished morphisms are unique up to $A^*$ conjugacy \cite[Proposition 3.4.3]{sakellaridis--venkatesh2017}. Knop and Schalke have proved the existence of distinguished morphisms in full generality \cite{knop--schalke2017}.
\end{rmk}

\begin{defn}
An $A$-parameter  $\psi: \lf \times \SL(2,\C) \rightarrow G^\vee$ is $X$-distinguished if it factors through the distinguished morphism $\xi: G_X^\vee \times \SL(2,\C) \rightarrow G^\vee$, that is, there exists a tempered (that is, bounded on $\wf$) $L$-parameter $\psi_X: \lf \rightarrow G_X^\vee$ such that $\psi(w,g) = \xi(\psi_X(w),g)$.
\end{defn}

We are now in a position to state the conjecture of Sakellaridis and Venkatesh on the parameters of relative discrete series representations (see \cite[Conjectures 16.2.2, 16.5.1]{sakellaridis--venkatesh2017} for statements of the full local conjectures).  Recall that an $L$-parameter $\phi: \lf \rightarrow G^\vee$ is elliptic if and only if the image of $\phi$ is not contained in any proper parabolic subgroup of $G^\vee$.

\begin{conj}[Sakellaridis and Venkatesh]\label{conj-sv}
	A relative discrete series representation $\pi$ in $L^2(X)$ is contained in an Arthur packet corresponding to an $X$-distinguished $A$-parameter $\psi: \lf \times \SL(2,\C) \rightarrow G^\vee$ such that the $L$-parameter $\psi_X: \lf \rightarrow  G_X^\vee$ is  elliptic.
\end{conj}

 Our ultimate aim is to prove that the relative discrete series representations for $X = \GL_n(F) \times \GL_n(F) \backslash \GL_{2n}(F)$ considered in \cite{smith2018} satisfy \Cref{conj-sv}.
 In the case that $\G = \GL_{2n}$, Arthur packets are $L$-packets and are singleton sets. This fact greatly simplifies the situation; moreover, we ultimately only need to consider $L$-parameters because the relative discrete series produced in \cite{smith2018} are all tempered. 
 
\section{Linear periods and relative discrete series}\label{sec-rds}
For the remainder of the paper, we let $G = \GL_{2n}(F)$ and let $H = \GL_{n}(F) \times \GL_{n}(F)$.
Let $X=H \backslash G$ be the $p$-adic symmetric space associated to the symmetric pair $(G,H)$.
In this section, we recall the main result of \cite{smith2018} and some background on local linear periods.
For the symmetric pair $(G,H)$, multiplicity-one is known due to work of Jacquet and Rallis \cite{jacquet--rallis1996}.

\begin{thm}[Jacquet--Rallis]
Let	$(\pi,V)$ be an irreducible admissible representation of $G$. The dimension of the space $\Hom_H(\pi,1)$ is at most one.  If $\dim \Hom_H(\pi,1) = 1$, then $\pi$ is self-contragredient, that is, $\wt \pi \cong \pi$.
\end{thm}

\begin{rmk}
Jacquet and Rallis also proved that if an irreducible admissible representation $(\pi,V)$ of $G$ admits a nonzero Shalika model, then $\pi$ is $H$-distinguished.
The converse for irreducible supercuspidal representations was obtained by Jiang, Nien, and Qin \cite[Theorem 5.5]{jiang--nien--qin2008}.
Sakellaridis and Venkatesh \cite[Section 9.5]{sakellaridis--venkatesh2017}, and independently Matringe \cite[Theorem 5.1]{matringe2014a}, proved the converse for irreducible $H$-relatively integrable and relative discrete series representations.
\end{rmk}

As above, for a positive integer $m$ we write $G_{m} = \GL_m(F)$ and, if $m$ is even, $H_{m} = \GL_{m/2}(F) \times \GL_{m/2}(F)$.
Let $\pi$ be a discrete series representation of $G_m$.
Let $L(s,\pi\times \pi)$ be the local Rankin--Selberg convolution $L$-function.
Shahidi \cite[Lemma 3.6]{shahidi1992} proved the local identity
\begin{align}\label{eq-RS-L-fcn-factors}
L(s,\pi\times \pi) = L(s,\pi, \wedge^2) L(s,\pi, \Sym^2),
\end{align}
where $L(s,\pi, \wedge^2)$, respectively $L(s,\pi, \Sym^2)$, denotes the exterior square, respectively symmetric square, $L$-function of $\pi$ defined via the Local Langlands Correspondence (LLC).
The $L$-function $L(s,\pi\times\pi)$ has a simple pole at $s=0$ if and only if $\pi$ is self-contragredient \cite{jacquet--piatetskii-shapiro--shalika1983}.
Note that $L(s,\pi, \wedge^2)$ cannot have a pole when $m$ is odd (see, for instance, \cite[Section 9]{bump2004}, \cite[Theorem 4.5]{kewat--raghunathan2012}).
For all discrete series representations of $G_m$, and all irreducible generic representations of $G_{2m}$, the associated Jacquet--Shalika  and Langlands--Shahidi local exterior square $L$-functions agree with with the exterior square $L$-function defined via the LLC \cite[Theorem 4.3 in $\S$4.2]{henniart2010}, \cite[Theorems 1.1 and 1.2]{kewat--raghunathan2012}.

Matringe proved the following results which characterize the $H$-distinction of discrete series representations of $G$.  The first result appears as {\cite[Proposition 6.1]{matringe2014a}}, and the second appears as {\cite[Theorem 6.1]{matringe2014a}}.

\begin{thm}[Matringe]\label{thm-matringe-ext-square}
Suppose that $\pi$ is a square integrable representation of $G$, then $\pi$ is $H$-distinguished
if and only if the exterior square $L$-function $L(s,\pi,\wedge^2)$ has a pole at $s=0$.
\end{thm}

\begin{thm}[Matringe]\label{thm-matringe-lin}
Suppose that $m= kr$ is even.  Let $\rho$ be an irreducible supercuspidal representation of $G_r$.  
Let $\pi = \st(k,\rho)$ be a generalized Steinberg representation of $G_m$.
\begin{enumerate}
\item If $k$ is odd, then $r$ must be even, and $\pi$ is $H_m$-distinguished
if and only if $L(s, \rho, \wedge^2)$ has a pole at $s=0$ if and only if $\rho$ is $H_r$-distinguished 
\item If $k$ is even, then $\pi$ is $H_m$-distinguished
if and only if $L(s, \rho, \Sym^2)$ has a pole at $s=0$. \label{thm-matringe-lin-even}
\end{enumerate}
\end{thm}

The author has studied relative discrete series for several $p$-adic symmetric spaces and carried out a systematic construction of relative discrete series in the papers \cite{smith2018b,smith2018}.  The following result forms part of \cite[Theorem 6.3]{smith2018}.  As the work \cite{smith2018} relies on \cite{kato--takano2008}, we further assume that the residual characteristic of $F$ is odd.

\begin{thm}\label{thm-smith}
Suppose that $F$ has odd residual characteristic.
Let $(m_1,\ldots,m_d)$ be a partition of $2n$ such that each $m_i$, $1\leq i \leq d$ is an even integer.
Let $\tau_i$, $1 \leq i \leq d$, be pairwise inequivalent $H_{m_i}$-distinguished discrete series representations of $G_{m_i}$.
The parabolically induced representation $\pi = \tau_1 \times \ldots \times \tau_d$ is a relative discrete series representation. 
That is, $\pi$ occurs in the discrete spectrum of $X=H\backslash G$.
\end{thm}

The aim of the current work is to prove that the local parameters of the relative discrete series representations described in \Cref{thm-smith} satisfy \Cref{conj-sv}.

\section{$X$-distinguished $X$-elliptic parameters}\label{sec-parameters}
Next we record a description of the dual group $G_X^\vee$ attached to $X = H\backslash G$ and choose a representative for the distinguished morphism $\xi: G_X^\vee \times \SL(2,\C) \rightarrow G^\vee$.
The existence of the distinguished morphism is now known in full generality by the work of Knop and Schalke \cite{knop--schalke2017}.
\Cref{lem-dual-morph} is well known, thus we omit the proof which 
 amounts to a routine calculation of the restricted root system attached to $X$ (see \cite[Sections 3.1 and 5.1]{smith2018}, \textit{cf.}~\cite[Table 3]{knop--schalke2017}).    We note that the fact that the distinguished morphism $\xi: G_X^\vee \times \SL(2,\C) \rightarrow G^\vee$ is trivial on the $\SL(2,\C)$-factor follows from the fact that a minimal $\theta$-split parabolic subgroup of $G$ is a Borel subgroup $P_0$, with $\theta$-stable Levi $A_0 = M_0$,  and the principal morphism of $\mathrm{SL}(2,\C)$ into $M_0^\vee = A_0^*$ is trivial (\textit{cf}.~\Cref{def-dist-mor}).

\begin{lem}\label{lem-dual-morph}
The complex dual group $G_X^\vee$ associated to $X$ is the symplectic group $\mathrm{Sp}(2n,\C)$.  The distinguished morphism $\xi: G_X^\vee \times \SL(2,\C) \rightarrow G^\vee$ is trivial on the $\SL(2,\C)$ factor and is given by the inclusion map into $G^\vee = \mathrm{GL}(2n,\C)$ on the $G_X^\vee$ component.
\end{lem}

Since the distinguished morphism $\xi: G_X^\vee \times \SL(2,\C) \rightarrow G^\vee$ is trivial on the $\SL(2,\C)$-factor, \Cref{conj-sv} predicts that an $X$-distinguished $A$-parameter $\psi: \lf\times \SL(2,\C)\rightarrow G^\vee$ must also be trivial on the Arthur $\SL(2)$.  That is, \Cref{conj-sv} predicts that the relative discrete series representations for $X$ should be tempered representations of $G$.
This is in agreement with the fact that the symmetric space $X$ is known to be tempered, that is, the Plancherel measure of $L^2(X)$ is supported on tempered representations of $G$ \cite[Proposition 2.7.1]{beuzart-plessis2018a}, \cite[Corollary 1.2]{gurevich--offen2016}.
Moreover, we may restrict our attention to tempered $L$-parameters for $G$ instead of dealing with more general Arthur parameters.

Let $\phi: \lf \rightarrow G^\vee$ be an $X$-distinguished $L$-parameter, i.e., $\phi\otimes 1 : \lf \times \SL(2,\C) \rightarrow G^\vee$ is an $X$-distinguished $A$-parameter. 
Thus $\phi$ is an $L$-parameter for $G$ that has image contained in $G_X^\vee = \mathrm{Sp}(2n,\C)$ (up to conjugacy), and since $\SO(2n+1)^\vee = \mathrm{Sp}(2n,\C)$, 
 we may regard $\phi$ as an $L$-parameter for the split special orthogonal group $\SO_{2n+1}(F)$.
We say that the parameter $\phi$ is $X$-elliptic if the image of $\phi$ is not contained in any proper parabolic subgroup of $G_X^\vee$.
In particular, if $\phi$ is $X$-elliptic, then the parameter $\phi : \lf \rightarrow G_X^\vee$ corresponds to a finite set (an $L$-packet) $\Pi_\phi$ of essential discrete series representation of  $\SO_{2n+1}(F)$ \cite[Proposition 6.6.5]{Arthur-book}.
\Cref{conj-sv} predicts that the discrete spectrum of $X$ is contained in the image of the local Langlands functorial transfers, determined by the inclusion $\mathrm{Sp}(2n,\C) \hookrightarrow \mathrm{GL}(2n,\C)$, of representations in the discrete spectrum of $\SO_{2n+1}(F)$.
In fact, it is known that an irreducible tempered representation $\pi$ of $G$ is $H$-distinguished if and only if $\pi$ is the local Langlands functorial lift of a generic discrete series representation of $\SO_{2n+1}(F)$ \cite[Theorem 2.1]{jiang--soudry2004}, \cite[Theorem 3.1]{matringe2015b}.

Let $\tau = \st(k,\rho)$ be a generalized Steinberg representation, where $\rho$ is an irreducible unitary supercuspidal representation of $G_r$.  Let $m = kr$. 
By the LLC for the general linear group \cite{Harris--Taylor-book, henniart2000, scholze2013}, the $L$-parameter $\phi_{\tau}: \lf \rightarrow \mathrm{GL}(m,\C)$ of the generalized Steinberg representation $\tau = \st(k,\rho)$ is equal to $\phi_{\tau} = \phi_{\rho} \otimes \style{S}(k)$, where $\phi_{\rho}: \wf \rightarrow \mathrm{GL}(r,\C)$ is the 
(irreducible) $L$-parameter of the supercuspidal representation $\rho$ and $\style{S}(k)$ is the unique irreducible $k$-dimensional representation of $\SL(2,\C)$.
The following proposition is a consequence of \Cref{thm-matringe-ext-square} \cite[Theorem 6.1]{matringe2014a}.

\begin{Prop}\label{prop-dist-gen-st-parameter}
Let $k,r\geq 2$ be integers such that $m=kr$ is even.  Let $\rho$ be an irreducible self-contragredient supercuspidal representation of $G_r$.  Let $\tau = \st(k,\rho)$ be the generalized Steinberg representation of $G_{m}$ attached to $k$ and $\rho$.  If $\tau$ is $H_{m}$-distinguished, then the image of the $L$-parameter $\phi_\tau$ is contained in the complex symplectic group $\mathrm{Sp}(m,\C)$.
\end{Prop}

\begin{proof}
As above, the $L$-parameter of $\tau$ is equal to $\phi_\tau = \phi_\rho \otimes \style{S}(k)$.

\begin{enumerate}
	\item If $k$ is odd, then $r$ is even.  By assumption $\tau$ is $H_{m}$-distinguished; therefore, by \Cref{thm-matringe-lin}, $\rho$ is $H_{r}$-distinguished.  By \cite[Theorem 5.5]{jiang--nien--qin2008}, $\rho$ is a local Langlands functorial transfer from $\mathbf{SO}_{r+1}(F)$.  In particular, $\rho$ is of symplectic type and its $L$-parameter $\phi_\rho$ has image contained in $\mathrm{Sp}(r,\C)$.  When $k$ is odd, the image of $\SL(2,\C)$ under $\style{S}(k)$ is contained in the complex orthogonal group $\mathrm{O}(k,\C)$.  It follows that the image of $\phi_\tau$ preserves a non-degenerate skew-symmetric bilinear form on $\C^r \otimes \C^k \cong \C^m$ given by the tensor product of the non-degenerate skew-symmetric bilinear form on $\C^r$ preserved by $\ran \phi_\rho$ and the non-degenerate symmetric bilinear form on $\C^k$ preserved by $\ran \style{S}(k)$. That is, $\ran \phi_\tau$ is contained in $\mathrm{Sp}(m,\C)$.
	\item If $k$ is even, then the image of $\SL(2,\C)$ is contained in the symplectic group $\mathrm{Sp}(k,\C)$.  Since $\tau$ is $H_{m}$-distinguished; it follows from \Cref{thm-matringe-ext-square}, that the symmetric square $L$-function $L(s,\rho, \Sym^2)$ has a pole at $s=0$.  In this case, $\rho \cong \wt \rho$ is self-contragredient and the exterior square $L$-function $L(s,\rho,\wedge^2)$ does not have a pole at $s=0$ \cite{jacquet--piatetskii-shapiro--shalika1983}, \cite[Lemma 3.6]{shahidi1992}.  It follows that $\phi_\rho$ is an irreducible self-dual $r$-dimensional representation of $\wf$ such that $\ran \phi_\rho$ is not contained in a symplectic group \cite[Theorem 5.5]{jiang--nien--qin2008}.  Following \cite[Section 1.2]{Arthur-book}, one may readily observe that $\ran \phi_\rho$ must then be contained in the complex orthogonal group $\mathrm{O}(r,\C)$.  As above, the image of $\phi_\tau$ preserves a non-degenerate skew-symmetric bilinear form on $\C^m$ and $\ran \phi_\tau$ is contained in $\mathrm{Sp}(m,\C)$.
\end{enumerate}
\end{proof}

\begin{thm}\label{thm-rds-dist-parameter}					
Let $\pi$ be an $H$-distinguished relative discrete series representation of $G$ constructed via \Cref{thm-smith}.
The image of the $L$-parameter $\phi_\pi$ of $\pi$ is contained in the symplectic group $\mathrm{Sp}(2n,\C)$; moreover, the image of $\phi_\pi$ is not contained in any proper parabolic subgroup of $\mathrm{Sp}(2n,\C)$.
\end{thm}

\begin{proof}
By assumption, $\pi = \tau_1 \times \ldots \times \tau_d$, where $2n = \sum_{i=1}^d m_i$ and each $m_i = k_ir_i$ is an even integer, and $\tau_i$ are pairwise inequivalent $H_{m_i}$-distinguished discrete series representation of $G_{m_i}$, for all $1\leq i \leq d$.  Moreover, each $\tau_i= \st(k_i,\rho_i)$ is a generalized Steinberg representation, where $\rho_i$ is an irreducible unitary self-contragredient supercuspidal representation of $G_{r_i}$.
The compatibility of the LLC with parabolic induction gives that $L$-parameter of $\pi$ is equal to the direct sum
\begin{align}\label{phi-pi-direct-sum}
	\phi_\pi &= \phi_{\tau_1} \oplus \ldots \oplus \phi_{\tau_d}.
\end{align}
By \Cref{prop-dist-gen-st-parameter}, $\ran \phi_{\tau_i} \subset \mathrm{Sp}(m_i,\C)$, for all $1\leq i \leq d$.
Each of the parameters $\phi_{\tau_i}$ are elliptic in $\mathrm{GL}(m_i,\C)$ and thus are also elliptic in $\mathrm{Sp}(m_i,\C)$ \cite[Lemma 3.1]{gurevich--offen2016}.
Moreover, since $\tau_i$ is determined by $\phi_{\tau_i}$ up to conjugacy in $\mathrm{GL}(m_i,\C)$, we may assume that $\ran \phi_{\tau_i}$ is contained in the symplectic group determined by the non-singular skew-symmetric matrix $J_{m_i}'$, where
\begin{align*}
J_{m}' & = \left( \begin{matrix}
 		0 & J_{m/2} \\
 		-J_{m/2} & 0 	
 \end{matrix} \right)
 & \text{and} & & J_k & =  \left( \begin{matrix} & & 1 \\ & \iddots & \\ 1 & & \end{matrix} \right),
\end{align*}
for any even integer $m\geq 2$ and any $k \geq 1$.
That is, for an even integer $m$, we realize the symplectic group $\mathrm{Sp}(m,\C)$ as the subgroup
\begin{align*}
\mathrm{Sp}(J_m') & = \{ g \in \mathrm{GL}(m,\C) : \tran{g} J_m' g = J_m' \}
\end{align*}
of $\mathrm{GL}(m,\C)$.
It follows that the image of the $L$-parameter $\phi_\pi$ is contained in the subgroup
	$\prod_{i=1}^d \mathrm{Sp}(J_{m_i}')$
of the symplectic group $\mathrm{Sp}(J_{2n}'') = \mathrm{Sp}(2n,\C)$ in $G^\vee = \mathrm{GL}(2n,\C)$, where $\mathrm{Sp}(J_{2n}'')$ is determined by the non-singular skew-symmetric matrix $J_{2n}'' = J_{m_1}' \oplus \ldots \oplus J_{m_d}'$.
In fact, $J_{2n}'' = w_+^{-1}J_{2n}'w_+$, where $w_+$ is the permutation matrix corresponding to the permutation of $2n$ given by $2i-1\mapsto i$, and $2i\mapsto 2n+1-i$, for $1\leq i \leq n$.
Thus $\mathrm{Sp}(J_{2n}'') = w_+^{-1} \mathrm{Sp}(J_{2n}') w_+$; moreover, $w_+ (\ran \phi_\pi) w_+^{-1} \subset \mathrm{Sp}(J_{2n}')$.\footnote{Recall that $\pi$ is determined by $\phi_\pi$ up to $G^\vee$-conjugacy.  The choice to work with $\mathrm{Sp}(2n,\C) = \mathrm{Sp}(J_{2n}')$ in what follows is convenient for working with block-upper triangular parabolic subgroups.}

It now remains to show that $\phi_\pi$ is elliptic in $\mathrm{Sp}(2n,\C)$.
We argue by contraction.  Suppose that $\phi_\pi$ factors through a (proper) maximal parabolic subgroup $P$ of $\mathrm{Sp}(2n,\C)$.  
Since each $\tau_i$ is a discrete series representation, the representations $\phi_{\tau_i}$ of $\lf$ are irreducible; in particular, $\phi_\pi$ is semisimple and must factor through a Levi component $L$ of $P$.
Up to conjugacy, $L$ has the form
\begin{align*}
L = \left\{ \left(\begin{matrix}	x & & \\ & y & \\ & & {J_m'}^{-1} \tran{x}^{-1} J_m' \end{matrix} \right) : x \in \mathrm{GL}(m,\C), y \in \mathrm{Sp}(2k,\C) \right \},
\end{align*}
for some integers $m, k$ so that $n=m+k$ and $m\geq 1$.
It follows that $\phi_\pi$ can be decomposed as the direct sum $\phi_\pi = \phi_1 \oplus \phi_0 \oplus \phi_1^\vee$, where $\phi_1$ and $\phi_0$ are finite dimensional representations of $\lf$, $\phi_1 \neq 0$, and $\phi_1^\vee$ denotes the contragredient of $\phi_1$.
By assumption, $\phi_\pi$ is the direct sum of self-dual non-isomorphic irreducible representations \eqref{phi-pi-direct-sum}; therefore, $\phi_\pi$ cannot be decomposed as $\phi_\pi = \phi_1 \oplus \phi_0 \oplus \phi_1^\vee$ and this is a contradiction.
  We conclude that $\ran \phi_\pi$ cannot be contained in a proper parabolic subgroup of $\mathrm{Sp}(2n,\C)$, that is, $\phi_\pi$ is elliptic in $\mathrm{Sp}(2n,\C)$.
\end{proof}

\Cref{thm-rds-dist-parameter} says precisely that the $L$-parameters of the known relative discrete series representations for $X=H\backslash G$ are $X$-distinguished and $X$-elliptic.
That is, the relative discrete series for  $X$ produced via \Cref{thm-smith} satisfy \Cref{conj-sv}.

\begin{rmk}
	 The last paragraph of the proof of \Cref{thm-rds-dist-parameter}, together with \Cref{conj-sv}, suggests that we cannot relax the regularity condition imposed in \Cref{thm-smith} (\textit{cf}.~\cite[Remark 6.6]{smith2018}). 
	 The author expects that the representations constructed in \Cref{thm-smith} exhaust the discrete spectrum of $X$, as predicted by \Cref{conj-sv}; however, the author does not have a proof of this fact. 
\end{rmk}

\subsection*{Acknowledgments}

The author is deeply grateful the anonymous referees for numerous helpful comments, corrections, and the suggested improvements to the proof of Theorem 3.4 that greatly increased its clarity.
Thank you also to Clifton Cunningham for helpful conversations on Arthur parameters and $L$-packets.


\providecommand{\bysame}{\leavevmode\hbox to3em{\hrulefill}\thinspace}
\providecommand{\MR}{\relax\ifhmode\unskip\space\fi MR }
\providecommand{\MRhref}[2]{%
  \href{http://www.ams.org/mathscinet-getitem?mr=#1}{#2}
}
\providecommand{\href}[2]{#2}

\end{document}